\documentclass[review,amsmath]{elsarticle}

\usepackage{amsmath,amsthm,amssymb,subeqnarray,amsfonts,graphics,color}
\usepackage{amsmath,amsthm,amssymb}
\usepackage{graphicx,subfigure}
\usepackage{tikz}
\usepackage{epstopdf}

\usepackage{setspace}
\usepackage{stmaryrd}
\usepackage{cases}
\usepackage{exscale}
\usepackage{relsize}

\usepackage{graphicx}
\usepackage{float}

\usepackage{booktabs}
\usepackage{threeparttable}
\usepackage{multirow}

\usepackage{algorithm}
\usepackage{algpseudocode}
\usepackage{amsmath}
\allowdisplaybreaks[4]

\usepackage[footnotesize]{caption2}

\numberwithin{equation}{section}
\numberwithin{table}{section}
\numberwithin{figure}{section}

\newtheorem{theorem}{Theorem}[section]

\newtheorem{remark}{Remark}[section]
\newtheorem{lemma}{Lemma}[section]

\biboptions{numbers,sort&compress}

\journal{}
\begin{document}

\begin{frontmatter}
\title{The compactness of Moser-Trudinger functionals with conical metric in the unit ball}
\tnotetext[t1]{ This work is partially supported by the NSF of China under grant Nos. 12001466 and  U19A2079.}

\author[a,b]{Qi Xia \corref{cor1}}
\ead{xq19991231@mail.dlut.edu.cn}

\author[a]{Yufeng Lu}
\ead{luf@dlut.edu.cn}

\address[a]{School of Mathematical Sciences,
Dalian University of Technology, Dalian, 116024, China}

\address[b]{DUT-BSU Joint Institute,
Dalian University of Technology, Dalian, 116024, China}

\cortext[cor1]{Corresponding author}

\begin{abstract}
Let $\mathbb{B}$ be the unit ball in $\mathbb{R}^2$, $W_0^{1,2} \left( \mathbb{B} \right)$ is a standard Sobolev space. Suppose a function $h_{\epsilon}(x)$ is radially symmetric, nonnegative, continuous on $\overline{\mathbb{B}}$ and satisﬁes $\underset{x \rightarrow 0}{\lim} h_{\epsilon}(x) |x|^{- 2 \epsilon} =1 $, with $h_{\epsilon} (x) >0$ on $\overline{\mathbb{B}} \setminus \{0\}$. In \citep{26}, Zhang proved that the supremum in the following inequality can be attained by some function $u_{\epsilon}$, i.e. ,  
\begin{flalign}
	\int_{ \mathbb{B} } h_{\epsilon} (x) e^{ 4 \pi \left(1 + \epsilon \right) {u_{\epsilon}}^2 }  dx = \underset{u \in W_0^{1,2} \left( \mathbb{B} \right) \cap \mathcal{S} \setminus \{0\} , ~ \int_{ \mathbb{B} } |\nabla u|^2 dx \leq 1}{\sup} \int_{\mathbb{B}} h_{\epsilon} (x) e^{4 \pi (1 + \epsilon) u^2 } dx, \label{eq: 0.1}
\end{flalign}
where $4 \pi$ is the best constant in the classical Moser-Trudinger inequality, and $\mathcal{S}$ is the set of radially symmetric functions.
In this paper, we consider the compactness of the sequence $\{ u_{\epsilon} \}_{\epsilon} $ and prove that the limit of this sequence is a function $u_0 \in C^1 \left(\overline{ \mathbb{B}} \right)$. Moreover, the $u_0$ is an extremal function of the
supremum
\begin{flalign*}
	\underset{u \in W_0^{1,2} \left( \mathbb{B} \right) \cap \mathcal{S} \setminus \{0\} , ~ \int_{ \mathbb{B} } |\nabla u|^2 dx \leq 1}{\sup} \int_{\mathbb{B}} e^{4 \pi u^2 } dx.
\end{flalign*}
\end{abstract}


\end{frontmatter}

\section{Introduction}
Let $\Omega$ be a smooth bounded domains in $\mathbb{R}^2$, and $ W_0^{1,2} \left( \Omega \right)$ is a standard Sobolev space with norm
\begin{flalign*}
	|| \cdot ||_{W_0^{1,2} \left( \Omega \right) } := \left( \int_{\Omega} |\nabla \cdot|^2 dx \right)^{ \frac{1}{2} },
\end{flalign*}
where $\nabla$ is the gradient operator. Then the classical Moser-Trudinger is
\begin{flalign*}
	\underset{u \in W_0^{1,2} \left( \Omega \right) \setminus \{0\}, ~ ||u||_{ W_0^{1,2} \left( \Omega \right) } \leq 1 }{\sup} \int_{\Omega} e^{ \alpha |u|^2 } dx < \infty.
\end{flalign*} 
For more details and  the origin of the Moser-Trudinger inequality, see \citep{20,21}. This inequality is sharp in the sense that, for every $\alpha > 4 \pi$, all integrals above are finite, but the supremum is not, for instance, \cite{17}, while the existence of extremal functions was established in \citep{4,5}.

In reference \citep{19,27}, Csat\'{o} and Yang  extended the sharp form of the Moser-Trudinger inequality and proved the existence of extremal functions $u_{\beta} \in W_0^{1,2} \left( \Omega \right) \cap C^0 \left( \overline{\Omega } \right) \cap C^1_{loc} \left( \Omega \right)$ as follows,
\begin{flalign*}
	\int_{\Omega} \frac{ e^{4 \pi \left( 1 - \beta \right) u_{\beta}^2 } }{|x|^{2 \beta} } dx = \underset{u \in W_0^{1,2} \left( \Omega \right) \setminus \{0\}, ~ \int_{\Omega} | \nabla u|^2 dx \leq 1 }{\sup} \int_{\Omega} \frac{ e^{4 \pi \left( 1 - \beta \right) u^2} }{|x|^{2 \beta} } dx, ~ \beta >0.
\end{flalign*} 
However, the existence of extremal functions cannot be extended to $\beta <0$. On the other hand, 
from the work of Calanchi and Terraneo \citep{6}, the following result holds, 
\begin{flalign}
	\underset{u \in W_0^{1,2} \left( \mathbb{B} \right) \bigcap \mathcal{S} \setminus \{0 \}, ||u||_{ W_0^{1,2} \left( \mathbb{B} \right) } \leq 1 }{\sup} \int_{\mathbb{B}} |x|^{ 2 \epsilon } e^{4 \pi \left( 1 +\epsilon \right) u^2 } dx < \infty, \label{eq: 1.1}
\end{flalign}
where $\mathcal{S} $ denotes the set of all radially symmetric functions. Also, in the references \cite{7,8,16}, Calanchi, Terraneo and others proved that there exist extremal functions $ u \in W_0^{1,2} \left( \mathbb{B} \right) \bigcap \mathcal{S}$ such that
\begin{flalign}
	\underset{u \in W_0^{1,2} \left( \mathbb{B} \right)  \setminus \{0 \} \bigcap \mathcal{S}, ||u||_{ W_0^{1,2} \left( \mathbb{B} \right) } \leq 1 }{\sup} \int_{\mathbb{B}} |x|^{ 2 \epsilon } e^{4 \pi \left( 1 + \epsilon \right) u^2 } dx = \int_{\mathbb{B}} |x|^{ 2 \epsilon } e^{4 \pi \left( 1 + \epsilon \right) u_{\epsilon}^2 } dx. \label{eq: 1.2}
\end{flalign}

In \citep{1}, Yang and Zhu extended equation \eqref{eq: 1.1} to high dimensions.
Equation \eqref{eq: 1.2} also holds in high dimensions. See \citep{2}. Moreover, in \citep{26}, Zhang extended the conical metric $|x|^{2 \epsilon}$ to a family of general functions $h_{\epsilon}(x)$ mentioned in the abstract. Also, the supremum of the inequality can be attained by a radially symmetric, nonnegative function, which is also continuous on the disk. In \citep{26}, Zhang proved that there exists a function that attains the supremum of the functional,
\begin{flalign*}
	\int_{\mathbb{B}} h_{\epsilon} (x) e^{ 4 \pi ( 1+ \epsilon) u^2 } dx,
\end{flalign*}
where $h_{\epsilon}(x)$ satisfies $\underset{\epsilon \rightarrow 0^+}{\lim} h_{\epsilon}(x) |x|^{ -2 \epsilon} =1$. 

Recently, Yang and Wang considered the compactness of extremals for critical singular Trudinger-Moser 
functionals in \cite{18}. Meanwhile, this work also provided a more delicate method to analyze the behavior of asymptotic symmetry than the method in \cite{19}. In \cite{3}, Shan and Li considered the Moser-Trudinger functional with the conical metric in the unit ball of $\mathbb{R}^2$, defined as follows,
\begin{flalign*}
	TM_{\epsilon} (u) := \int_{\mathbb{B} } e^{ 4 \pi \left( 1+ \epsilon \right) u^2 } dg.
\end{flalign*}
They obtained that there exist radially symmetric functions $u_{\epsilon}$ satisfying
\begin{flalign*}
	TM_{\epsilon} \left( u_{\epsilon} \right) = \underset{ u \in W_0^{1,2} \left( \mathbb{B} \right) \setminus \{0\}, ~ ||u||_{ W_0^{1,2} \left( \mathbb{B} \right) } \leq 1 }{\sup} TM_{\epsilon} (u)
\end{flalign*}
for any $\epsilon >0$. Meanwhile,
these functions are also extremal functions by a rearrangement argument. Furthermore, the maximizer sequence $\{ u_{\epsilon} \}_{\epsilon}$ converges as the conical
metric $g$ converges to the standard Euclidean metric. Then we consider whether these results hold in a more general family of functions. 
The principal contributions of this work are outlined below,
\begin{theorem}
	 Assume that $u_{\epsilon}$ is a sequence of maximizers for the supremum in \eqref{eq: 0.1}. Suppose a function $h_{\epsilon}(x)$ is radially symmetric, nonnegative,
	 continuous on $\overline{\mathbb{B}}$ and satisﬁes $h_{\epsilon} (x) > 0$ on $\overline{ \mathbb{B} } \setminus \{0\}$ and
	 \begin{flalign*}
	 	& \underset{\epsilon \rightarrow 0}{\lim} h_{\epsilon} (x) = |x|^{2 \epsilon} \cdot h_0 (x),\\
		& \underset{x \rightarrow 0}{\lim} h_{\epsilon} (x) |x|^{ - 2 \epsilon} =1.
	\end{flalign*} 
	 Then, up to a subsequence, there exists a function $u_0$ such that $u_{\epsilon } \rightarrow u_0$ in $C^1 ( \overline{\mathbb{B}} )$ and $u_0$ is an extremal function of the following supremum
	\begin{flalign}
		\underset{u \in W_0^{1, 2} ( \overline{\mathbb{B}} ) \cap \mathcal{S} \setminus \{0\} , ~ ||u||_{W_0^{1, 2} ( \overline{\mathbb{B}}) }  \leq  1}{\sup} \int_{\mathbb{B} } h_{0} (x) e^{ 4 \pi u^{2} }  dx. \label{1.3}
	\end{flalign}
\end{theorem}

The outline of the proof of Theorem 1.1 is given below. Assume that $u_{\epsilon} $ are as mentioned in equation \eqref{1.3}. They are also the solutions of Euler--Lagrange equations below. If $u_{\epsilon}$ are not bounded, then there exists a constant $A_0$ satisfying
\begin{flalign*}
	\underset{u \in W_0^{1,2} \left( \mathbb{B} \right) \setminus \{0 \} \cap \mathcal{S}, ||u||_{ W_0^{1,2} \left( \mathbb{B} \right) } \leq 1 }{\sup} \int_{\mathbb{B}} h_{\epsilon}(x) e^{4 \pi u^{2} } dx  \leq  \int_{\mathbb{B}} h_0(x) dx + e^{ 1 + 4 \pi A_0  }.
\end{flalign*} 
On the other hand, we also find that there exist functions $u$ such that
\begin{flalign*}
	\underset{u \in W_0^{1,2} \left( \mathbb{B} \right) \setminus \{0 \} \cap \mathcal{S}, ||u||_{ W_0^{1,2} \left( \mathbb{B} \right) } \leq 1 }{\sup} \int_{\mathbb{B}} h_{\epsilon}(x) e^{4 \pi u^{2} } dx  >  \int_{\mathbb{B}} h_{0}(x) dx + e^{ 1 + 4 \pi A_0  }.
\end{flalign*} 
The above two inequalities yield a contradiction, which in turn implies the uniform boundedness of $u_{\epsilon}$ on $\mathbb{B}$. Then,  applying elliptic estimates, we obtain the desired result immediately.

\section{Some Lemmas}

Theorem 1.1 is proved by means of a blow-up analysis analogous to the one developed in \cite{13,14, 22, 23}, and we will divide the proof into several steps.

\subsection{The Euler--Lagrange equation of $u_{\epsilon}$}

For $\epsilon > 0$, we write 
\begin{flalign*}
	\mathcal{J} = \underset{u \in W_0^{1,2} ( \overline{\mathbb{B}} ), ~ ||u||_{W_0^{1, 2} ( \overline{\mathbb{B}}) }  \leq  1}{\sup} \int_{\mathbb{B} } h_0 (x) e^{ 4 \pi u^2 }  dx,
\end{flalign*}
and
\begin{flalign*}
	\mathcal{J_{\epsilon}} = \underset{u \in W_0^{1,2} ( \overline{\mathbb{B}} ), ~ ||u||_{W_0^{1,2} ( \overline{\mathbb{B}}) }  \leq  1}{\sup} \int_{\mathbb{B} } h_{\epsilon} (x) e^{ 4 \pi (1 + \epsilon) u^{2 } }  dx,
\end{flalign*}
From equation \eqref{1.3}, there exist some $u_{\epsilon} \in W_0^{1,2}\left(\mathbb{B} \right)$ satisfying $|| \nabla u_{\epsilon}||_{L^2 (\mathbb{B})} = 1$ and 
\begin{flalign*}
	\mathcal{J_{\epsilon}} = \int_{\mathbb{B} } h_{\epsilon}(x) e^{ 4 \pi (1 + \epsilon) u_{\epsilon}^{2 } }  dx.
\end{flalign*}

It is not difficult to find that $u_{\epsilon}$ satisfy the Euler--Lagrange equations
\begin{equation}
	\begin{cases}
		& - \Delta u_{\epsilon} = \frac{1}{\lambda_{\epsilon}} h_{\epsilon} (x) u_{\epsilon}  e^{ 4 \pi (1 + \epsilon) u_{\epsilon}^{2} }, ~ \text{in} ~ \mathbb{B};\\
		& u_{\epsilon} \geq 0, ||u_{\epsilon}||_{W_0^{1,2} (\mathbb{B})} =1;\\
		& \lambda_{\epsilon} = \int_{\mathbb{B}} h_{\epsilon} (x) u_{\epsilon}^{2}  e^{ 4 \pi (1 + \epsilon) u_{\epsilon}^{2}  }dx. \label{eq: 2.1}
	\end{cases}	
\end{equation}
Since $u_{\epsilon}$ is bounded in $W_0^{1,2} (\mathbb{B})$, there exists  a subsequence (still denoted by $\{u_{\epsilon}\}_{\epsilon}$)
\begin{equation}
	\begin{cases}
		& u_{\epsilon} \rightharpoonup u_0, ~ W_0^{1,2} (\mathbb{B});\\
		& u_{\epsilon} \rightarrow u_0, ~ L^p (\mathbb{B}), ~ p>1;\\
		&  u_{\epsilon} \rightarrow u_0, ~ \text{a.e.} ~ \mathbb{B},
	\end{cases} \label{eq: 2.2}
\end{equation}
 as $\epsilon \rightarrow 0$. By the Lebesgue dominated convergence theorem and Fotou's lemma, we have
\begin{flalign*}
	\int_{\mathbb{B}} e^{ 4 \pi u^{2}  } dx & \leq \underset{ \epsilon \rightarrow 0}{\lim \inf} \mathcal{J}_{\epsilon} \\
	& = \underset{ \epsilon \rightarrow 0}{\lim \inf} \int_{\mathbb{B}} h_{\epsilon} (x) e^{ 4 \pi (1 + \epsilon) u_{\epsilon}^{2 }  } dx. 
\end{flalign*}
Then we obtain 
\begin{flalign*}
	\mathcal{J} \leq \underset{ \epsilon \rightarrow 0}{\lim \inf} \int_{\mathbb{B}} h_{ \epsilon} (x) e^{ 4 \pi (1 + \epsilon) u_{\epsilon}^{2}  } dx. 
\end{flalign*}
Next, we prove that $\lambda_{\epsilon} > 0$ has a lower bound. Since $t e^{t} \geq e^{t  } - 1$ for $t \geq 0$, there holds 
\begin{flalign*}
	& \lambda_{\epsilon} = \int_{\mathbb{B}} h_{\epsilon}(x) u_{\epsilon}^{2}  e^{ 4 \pi (1 + \epsilon) u_{\epsilon}^{2}  }dx\\
	& \geq \frac{1}{4 \pi (1+ \epsilon)} \int_{\mathbb{B}} 4 \pi (1+ \epsilon) h_{ \epsilon} (x) u_{\epsilon}^{2}  e^{ 4 \pi (1 + \epsilon) u_{\epsilon}^{ 2 }  }dx\\
	& \geq  \frac{1}{4 \pi (1+ \epsilon)} \int_{\mathbb{B}} h_{ \epsilon} (x) \left(e^{ 4 \pi (1 + \epsilon) u_{\epsilon}^{2}  } - 1\right)dx.
\end{flalign*}
From then on, as $\epsilon \rightarrow 0$ and $e^{4 \pi (1 + \epsilon) u_{\epsilon}^2} - 1>0$, we get that 
\begin{flalign*}
	\underset{\epsilon \rightarrow 0}{\lim \inf} \lambda_{\epsilon} > 0; 
\end{flalign*}

We let $c_{\epsilon} = u_{\epsilon} (0) = \underset{\mathbb{B}}{\max}u_{\epsilon}$. There are two possibilities for the analysis of $u_{\epsilon}$, either $u_{\epsilon}$ are bounded in $\mathbb{B}$, or $\underset{\mathbb{B}}{\max}u_{\epsilon} \rightarrow + \infty$ as $\epsilon \rightarrow 0^+$. Then we describe the blow-up phenomenon of $u_{\epsilon}$.

\subsection{Blow-up Analysis}

 \begin{lemma}
 	The function $u_0 \equiv 0$ and $|\nabla u_{\epsilon}|^2 dx \rightharpoonup \delta_0$ in the sense of measure, where $\delta_0$ denotes the usual Dirac measure giving  unit mass to the point 0.
 \end{lemma}

\begin{proof}
	We denote that
	\begin{flalign*}
		f_{\epsilon} (x) =\frac{1}{\lambda_{\epsilon}} h_{ \epsilon} (x) e^{ 4 \pi (1 + \epsilon) u_{\epsilon}^{2} }.
	\end{flalign*}
    We can find that $\int_{\mathbb{B}} |f_{\epsilon}|^p dx \leq C < + \infty $ by Lion's concentration-compactness principle in \citep{24,25}. Or, 
    \begin{flalign*}
    	\int_{\mathbb{B}} f_{\epsilon}^p dx & = \frac{1}{\lambda_{\epsilon}^p} \int_{\mathbb{B}} h_{\epsilon}^p(x) e^{4 \pi (1 + \epsilon) p u_{\epsilon}^2 } dx\\
    	&  \leq \frac{1}{\lambda_{\epsilon}^p}\int_{\mathbb{B}} h_{\epsilon}^p(x) e^{4 \pi (1 + \epsilon) p \left( (1 + \eta) (u_{\epsilon}-u_0)^2 + (1 + \eta^{-1} ) u_0^2 \right) } dx\\
    	& \leq \frac{1}{\lambda_{\epsilon}^p} \left\lbrace \int_{\mathbb{B}} h_{\epsilon}^{p p_1}(x) e^{4 \pi (1 + \epsilon) p p_1 (1 + \eta) (u_{\epsilon}-u_0)^2  } dx \right\rbrace ^{ \frac{1}{p_1} }\\
    	& \times \left\lbrace \int_{\mathbb{B}} h_{\epsilon}^{p p_2}(x) e^{4 \pi (1 + \epsilon) p p_2 (1 + \eta^{-1} ) u_0^2  } dx\right\rbrace ^{ \frac{1}{p_2} },
    \end{flalign*}
    where $p>1$, $\eta >0$ and $ \frac{1}{p_1} + \frac{1}{p_2} =1 $. We denote that the first part is S and the other is T. Assume that $u_0 \neq 0$. Note that 
    \begin{flalign*}
    	||u_{\epsilon} - u_0||_{W_0^{1,2}}^2 & = \int_{\mathbb{B}} |\nabla u_{\epsilon}|^2 dx- 2 \int_{\mathbb{B}} |\nabla u_{\epsilon} \nabla u_0| dx + \int_{\mathbb{B}} |\nabla u_0|^2dx\\
    	&= 1- \int_{\mathbb{B}} |\nabla u_0|^2 dx + o_{\epsilon} (1).
    \end{flalign*}
    Let $p_1$, $p$ suﬃciently close to 1, and $\eta$ uﬃciently close to 0, $\epsilon >0$ suﬃciently small such that $p p_1 ||u_{\epsilon} - u_0||^2_{W_0^{1,2} } (1 + \eta) < 1-v $ for some $v>0$, then then by
    Trudinger–Moser inequality in the unit ball, we have
    \begin{flalign*}
    	& S \leq C_1,\\
    	& T \leq C_2.\\
    \end{flalign*}
    Then we have that 
    \begin{flalign*}
    	\int_{\mathbb{B}} f_{\epsilon}^p dx \leq C < +\infty.
    \end{flalign*}
    Applying elliptic estimates to equations, we obtain that $u_{\epsilon}$ is uniformly bounded on $\mathbb{B}$. This contradicts that $c_{\epsilon} \rightarrow \infty$ as $\epsilon \rightarrow 0^+$. Therefore $u_0 \equiv 0$.
        
    Next, We need to prove that $|\nabla u_{\epsilon}|^2 \rightharpoonup \delta_0$ as $\epsilon \rightarrow 0^+$.  Firstly, we have that $||u_{\epsilon}||_{W^{1,2} (\mathbb{B})} =1$. If the assertion is flase, then there would exist some $0 < r_0 <1$ such that 
        \begin{flalign*}
        	\underset{\epsilon \rightarrow 0}{\lim \sup} \int_{\mathbb{B}_{r_0}} |\nabla u_{\epsilon} |^2 dx \leq \gamma < 1.  
        \end{flalign*}
        It is not difficult to see that $\Delta u_{\epsilon}$ is bounded in $L^q (\mathbb{B}_{\frac{r_0}{2}})$ for some $q>1$. From elliptic estimates for the equation, we have that $u_{\epsilon}$ is bounded in $\mathbb{B}_{\frac{r_0}{4}}$, again contradicting $c_{\epsilon} \rightarrow + \infty$. We then obtain $|\nabla u_{\epsilon}|^2 \rightharpoonup \delta_0$ as $\epsilon \rightarrow 0$. Thus the lemma is proved.
\end{proof}

Let $r_{\epsilon} = \lambda_{\epsilon}^{\frac{1}{2}} c_{\epsilon}^{- 1 }e^{ -2 \pi (1+ \epsilon) c_{\epsilon}^{2} }$ and $\tilde{u}_{\epsilon} = \tilde{u}_{\epsilon} (x) : = u_{\epsilon} (r_{\epsilon}^{ \frac{1}{1+ \epsilon} } x)$, from \citep{1} and the classical Trudinger-Moser inequality, we know that 
\begin{flalign*}
	r_{\epsilon} e^{\alpha u_{\epsilon}^{ 2 } } \rightarrow 0
\end{flalign*} 
as $\epsilon \rightarrow 0$, $\forall \alpha < 2 \pi (1+\epsilon)$. For $x \in \mathbb{B}_{r_{\epsilon}^{-1}  }$, we define two sequences of functions by
\begin{flalign*}
	\psi_{\epsilon} (x) = c_{\epsilon}^{-1} u_{\epsilon} (r_{\epsilon}^{ \frac{1}{1+ \epsilon} } x) ,
\end{flalign*}
and 
\begin{flalign*}
	\phi_{\epsilon} (x) = c_{\epsilon}\left( u_{\epsilon} (r_{\epsilon}^{ \frac{1}{1+ \epsilon} } x) -c_{\epsilon} \right)
\end{flalign*}
for $x \in \mathbb{B}_{\epsilon} = \{ x \in \mathbb{R}^2: ~ r_{\epsilon}^{\frac{1}{1+\epsilon }} x \in \mathbb{B} \}$.
Then, we consider the asymptotic behavior of $u_{\epsilon}$ at the origin.

\begin{lemma}
	Up to a subsequence, there holds $\psi_{\epsilon} \rightarrow 1$ in $C^1_{loc} (\mathbb{R}^2)$ and $\phi_{\epsilon} \rightarrow \phi$ in $C^1_{loc} (\mathbb{R}^2)$ as $\epsilon \rightarrow 0$, where 
	\begin{flalign*}
		\phi(x) = - \frac{1}{4 \pi (1 + \epsilon) } \ln \left(1+ \frac{\pi}{1 + \epsilon} |x|^{2( 1+ \epsilon)} \right). 
	\end{flalign*}
\end{lemma}

\begin{proof}
	Direct calculation shows
    \begin{flalign*}
    	- \Delta \psi_{\epsilon} (x) = c_{\epsilon}^{-2}  r_{\epsilon}^{ \frac{-2 \epsilon}{1 + \epsilon} } h_{\epsilon} \left( r_{\epsilon}^{ \frac{1}{1 + \epsilon} } x \right) \psi_{\epsilon}  e^{ 4 \pi (1 + \epsilon) \left( \tilde{u}_{\epsilon}^2- c_{\epsilon}^2 \right) } 
    \end{flalign*}
     We also have
    \begin{flalign*}
    	 - \Delta \phi_{\epsilon} (x) 
    	=  r_{\epsilon}^{ \frac{-2 \epsilon}{1 + \epsilon} } h_{\epsilon} \left( r_{\epsilon}^{ \frac{1}{1 + \epsilon} } x \right) \psi_{\epsilon}  e^{ 4 \pi (1 + \epsilon) \left( \tilde{u}_{\epsilon}^2- c_{\epsilon}^2 \right) } .
    \end{flalign*}
     We note that $| \psi_{\epsilon} |\leq 1$, $r_{\epsilon} \rightarrow 0$, $c_{\epsilon} \rightarrow + \infty$ as $\epsilon \rightarrow 0$. Then we have $\Delta \psi_{\epsilon} \in L^p (\mathbb{B}_{\epsilon} )$ for some $p>1$. From elliptic estimates in \cite{9}, we get
     \begin{flalign*}
     	\psi_{\epsilon} \rightarrow \psi_0, ~ \text{in} ~ C^1_{loc} (\mathbb{R}^2) ~ \text{as} ~ \epsilon \rightarrow 0,
     \end{flalign*}
     where
     \begin{flalign*}
     	\psi_0 (0) = 1, ~ 0 \leq \psi_0 (x) \leq 1, ~ \forall x \in \mathbb{R}^N.
     \end{flalign*}
     Also, The function $\psi_0$ satisfies $- \Delta \psi_{0}  =0 $ in $\mathbb{R}^2$. From the Liouville theorem, we deduce that  $\psi_0 (0) \equiv \psi_0 (x) = 1$.
     
     Also we have by applying elliptic estimates in \cite{4} that,
     \begin{flalign*}
     	\phi_{\epsilon} \rightarrow \phi,~ \text{in} ~ C^1_{loc} (\mathbb{R}^2) ~ \text{as} ~ \epsilon \rightarrow 0,
     \end{flalign*}
     where $\phi$ satisﬁes
      \begin{equation*}
     	\begin{cases}
     		& - \Delta \phi = |x|^{2 \epsilon}e^{ 8 \pi (1 + \epsilon)\phi }, ~ \text{in} ~ \mathbb{R}^N;\\
     		& \underset{\mathbb{R}^2}{\sup} \phi = \phi(0) = 0.
     	\end{cases}
     \end{equation*}
     Thus the results hold.
\end{proof}

\begin{remark} \label{rem: 2.1}
    This lemma means that the weak function $\phi$ satisfies
    \begin{equation*}
    	\begin{cases}
    		& - \Delta \phi  = |x|^{2 \epsilon} e^{ 8 \pi (1 + \epsilon) \phi }, ~ \text{in} ~ \mathbb{R}^2;\\
    		& \underset{\mathbb{R}^2}{\sup} \phi = \phi(0) = 0.
    	\end{cases}
    \end{equation*}
     Moreover, $\phi(x) = - \frac{1}{4 \pi (1 + \epsilon) } \ln \left(1+ \frac{\pi}{1 + \epsilon} |x|^{2( 1+ \epsilon)} \right)$ and, 
     \begin{flalign}
     	\int_{\mathbb{R}^2} |x|^{ 2 \epsilon}e^{ 8 \pi (1 + \epsilon) \phi } dx = 1. \label{eq: 2.3}
     \end{flalign}
\end{remark}

We proceed to analyze the convergence of $u_{\epsilon} $ away from zero. As in \cite{12}, we set $u_{\epsilon, \iota} = \min \{ \iota c_{\epsilon}, ~ u_{\epsilon} \}$ for any $\iota \in (0,1) $.
 
\begin{lemma}
	For any $ 0 < \iota <1$, we have 
	\begin{equation*}
		\underset{\epsilon \rightarrow 0}{\lim} \int_{\mathbb{B}} |\nabla u_{\epsilon, \iota} |^2 dx = \iota.
	\end{equation*}
\end{lemma}

\begin{proof}
	By the divergence theorem, we have
     \begin{flalign*}
     	\int_{\mathbb{B}} |\nabla u_{\epsilon, \iota} |^2 dx & = -\int_{\mathbb{B}} u_{\epsilon, \iota} \Delta u_{\epsilon}  dx\\
     	& \geq \frac{1}{\lambda_{\epsilon}} \int_{\mathbb{B}_{ R_{\epsilon}^{ \frac{1}{ 1+ \epsilon } } } (0) }  h_{ \epsilon} (x) u_{\epsilon} \cdot u_{\epsilon,\iota}  e^{ 4 \pi (1 + \epsilon) u_{\epsilon}^2 } dx\\
     	& = \iota ( 1 + o_{\epsilon} (1))  \int_{\mathbb{B}_R (0)} h_{ \epsilon} (y) e^{  4 \pi ( 1+ \epsilon ) \left(u_{\epsilon} (y)^2 -  c_{\epsilon}^2 \right) } dy + o_{\epsilon}(1).
     \end{flalign*}
     Testing the equation \eqref{eq: 2.1} by $(u_{\epsilon} - \iota c_{\epsilon})^+$, and we get that 
     \begin{flalign*}
     	\int_{\mathbb{B}} |\nabla (u_{\epsilon} - \iota c_{\epsilon})^+ |^2 dx & = -\int_{\mathbb{B}} (u_{\epsilon} - \iota c_{\epsilon})^+ \Delta (u_{\epsilon} - \iota c_{\epsilon})^+  dx\\
     	& \geq \frac{1}{\lambda_{\epsilon}} \int_{\mathbb{B}_{ R_{\epsilon}^{ \frac{1}{ 1+ \epsilon } } } (0) }  h_{ \epsilon} (x) u_{\epsilon} \cdot ( u_{\epsilon} - \iota c_{\epsilon} )^+  e^{ 4 \pi (1 + \epsilon) u_{\epsilon}^2 } dx\\
     	& = (1 - \iota) ( 1 + o_{\epsilon} (1))  \int_{\mathbb{B}_R (0)} h_{ \epsilon} (y) e^{  4 \pi ( 1+ \epsilon ) \left(\tilde{u}_{\epsilon}^2 - c_{\epsilon}^2 \right) }  dy + o_{\epsilon}(1).
     \end{flalign*}
     In fact, let $t_{\epsilon}$ lie between $u_{\epsilon}(r_{\epsilon}^{ \frac{1}{1+ \epsilon} } x )$ and $c_{\epsilon}$. By the mean value theorem, we have that
     \begin{flalign*}
     	\tilde{u}_{\epsilon}^2 - c_{\epsilon}^2 = 2 t_{\epsilon} c_{\epsilon}^{-1} \phi_{\epsilon} (x) .
      \end{flalign*}
     From Fatou's Lemma and equation (2.3), we have
     \begin{flalign}
     	& \underset{\epsilon \rightarrow 0}{\lim \inf} \int_{\mathbb{B}} |\nabla u_{\epsilon, \iota}|^2 dx \geq \iota ; \label{eq: 2.4} \\
     	& \underset{\epsilon \rightarrow 0}{\lim \inf} \int_{\mathbb{B}} |\nabla (u_{\epsilon} - \iota c_{\epsilon} )|^2 dx \geq 1 - \iota . \label{eq: 2.5}
     \end{flalign}
     We also note that 
     \begin{flalign}
     	\int_{\mathbb{B}} |\nabla u_{\epsilon, \iota}|^2 dx + \int_{\mathbb{B}} | \nabla (u_{\epsilon} - \iota c_{\epsilon})^+ |^2 dx = \int_{\mathbb{B}} |\nabla u_{\epsilon}|^2 dx = 1+ o_{\epsilon} (1). \label{eq: 2.6}
     \end{flalign} 
     Combining \eqref{eq: 2.4}, \eqref{eq: 2.5} and \eqref{eq: 2.6}, we have that the results hold.
\end{proof}

\begin{lemma}
	There holds 
	\begin{flalign*}
		\underset{\epsilon \rightarrow 0}{\lim \sup}\int_{\mathbb{B}} h_{\epsilon} (x) e^{ 4 \pi (1 + \epsilon) u_{\epsilon}^2 } dx \leq \int_{\mathbb{B}} h_{0} (x) dx + \underset{R \rightarrow + \infty}{\lim} \underset{\epsilon \rightarrow 0}{\lim \sup} \int_{ \mathbb{B}_{R_{ r_{\epsilon}^{ \frac{1}{1+\epsilon} } }} } h_{\epsilon} (x) e^{ 4 \pi (1 + \epsilon) u_{\epsilon}^2 } dx.
	\end{flalign*}
\end{lemma}

\begin{proof}
	For any $0<\iota<1$, we have
	\begin{flalign*}
		\int_{\mathbb{B}} h_{\epsilon} (x) e^{ 4 \pi ( 1+ \epsilon) u_{\epsilon}^2  } = \int_{ u_{\epsilon} \leq \iota c_{\epsilon} } h_{\epsilon} (x) e^{ 4 \pi ( 1+ \epsilon) u_{\epsilon}^2}   dx + \int_{ u_{\epsilon} > \iota c_{\epsilon} } h_{\epsilon} (x) e^{ 4 \pi ( 1+ \epsilon) u_{\epsilon}^2  } dx.
	\end{flalign*}
    The first lemma implies that $u_{\epsilon,\iota}$ converges to $0$ a.e. in $\mathbb{B}$, from which we obtain
     \begin{flalign*}
     	\int_{ u_{\epsilon} \leq \iota c_{\epsilon} } h_{\epsilon} (x) e^{ 4 \pi ( 1+ \epsilon) u_{\epsilon}^2  } dx \leq \int_{ \mathbb{B} } h_{\epsilon} (x) e^{ 4 \pi ( 1+ \epsilon) u_{\epsilon, \iota}^2  } dx = \int_{\mathbb{B}} h_{\epsilon} (x) dx  + o_{ \epsilon } (1).
     \end{flalign*}
     In additional, we estimate
     \begin{flalign*}
     	\int_{ u_{\epsilon} > \iota c_{\epsilon} } h_{\epsilon} (x) e^{ 4 \pi ( 1+ \epsilon) u_{\epsilon}^2} dx & \leq \frac{1}{\iota^2 c_{\epsilon}^2 } \int_{ u_{\epsilon} > \iota c_{\epsilon} } h_{\epsilon} (x) u_{\epsilon}^2 e^{ 4 \pi ( 1+ \epsilon) u_{\epsilon}^2 } dx\\
     	& \leq \frac{1}{\iota^2  c_{\epsilon}^2}  \int_{ \mathbb{B} } h_{\epsilon} (x) u_{\epsilon}^{2} e^{ 4 \pi ( 1+ \epsilon) u_{\epsilon}^2  } dx\\
     	& = \frac{ \lambda_{\epsilon} }{ \iota^2 c_{\epsilon}^2 }.
     \end{flalign*}
      Letting $\iota \rightarrow 1$, then we have
      \begin{flalign}
      	\underset{ \epsilon \rightarrow 0}{ \lim \sup } \int_{\mathbb{B}} h_{\epsilon} (x) e^{ 4 \pi (1+ \epsilon) u_{\epsilon}^2 } dx \leq \int_{\mathbb{B}} h_{0} (x) dx + \underset{\epsilon \rightarrow 0}{\lim \sup} \frac{\lambda_{\epsilon}}{  c_{\epsilon}^2 }. \label{eq: 2.7}
      \end{flalign}
      On the other hand,  for any fixed $R > 0$, we obtain
      \begin{flalign*}
      	\int_{\mathbb{B}_{R r_{\epsilon}^{ \frac{1}{1+ \epsilon} } } } h_{\epsilon } (x) e^{ 4 \pi (1+ \epsilon) u_{\epsilon}^2 } dx & = \int_{\mathbb{B}_R } h_{\epsilon } (y) r_{\epsilon}^2 e^{ 4 \pi (1+ \epsilon) \tilde{u}_{\epsilon}^2 } dy\\
      	& = \lambda_{\epsilon} c_{\epsilon}^{- 2 } \cdot \int_{\mathbb{B}_R} |x|^{2 \epsilon} e^{ 8 \pi \phi } dx+o_{\epsilon}(1).
      \end{flalign*}
      Thus,
      \begin{flalign}
      	& \underset{R \rightarrow + \infty}{\lim } \underset{\epsilon \rightarrow 0}{\lim \sup} \int_{\mathbb{B}_{R r_{\epsilon}^{ \frac{1}{1+ \epsilon} } } } h_{\epsilon } (x) e^{ 4 \pi (1+ \epsilon) u_{\epsilon}^2 } dx  \notag\\
      	& = \underset{R \rightarrow + \infty}{\lim } \underset{\epsilon \rightarrow 0}{\lim \sup} \lambda_{\epsilon} c_{\epsilon}^{- 2 } \cdot \int_{\mathbb{B}_R} |x|^{2 \epsilon} e^{ 8 \pi \phi } dx \notag \\
      	& \geq \underset{\epsilon \rightarrow 0}{\lim \sup} \frac{\lambda_{\epsilon}}{  c_{\epsilon}^{2} }. \label{eq: 2.8}
      \end{flalign}
      The lemma follows from equations \eqref{eq: 2.7} and \eqref{eq: 2.8}.
\end{proof}

\begin{lemma}
	For any $\alpha < 2$, we have that
	\begin{flalign*}
		\underset{\epsilon \rightarrow 0}{\lim}  \frac{\lambda_{\epsilon}}{c_{\epsilon}^{\alpha}} = + \infty.
	\end{flalign*}
\end{lemma}

\begin{proof}
	If not, then, $\frac{\lambda_{\epsilon}}{c_{\epsilon}^{2}} \rightarrow 0$ as $\epsilon \rightarrow 0$. For any $u \in W_0^{1,2} ( \mathbb{B})$ with $|| u||_{ W_0^{1,2} (\mathbb{B}) } = 1$. According to the equations \eqref{eq: 2.1} and \eqref{eq: 2.7}, we have
	\begin{flalign*}
		\int_{\mathbb{B}} h_{\epsilon} (x) dx  \leq \int_{\mathbb{B}} h_{\epsilon}(x) e^{ 4 \pi u^{2}  } dx \leq \mathcal{J } & \leq \underset{\epsilon \rightarrow 0}{\lim \inf} \int_{\mathbb{B}} h_{\epsilon} (x) e^{ 4 \pi (1 + \epsilon) u_{\epsilon}^2  } dx\\
		& \leq \int_{\mathbb{B}} h_{\epsilon} (x) dx  + \underset{\epsilon \rightarrow 0}{\lim \sup} \frac{\lambda_{\epsilon}}{c_{\epsilon}^{ 2}} =\int_{\mathbb{B}} h_{\epsilon} (x) dx .
	\end{flalign*}
    It is impossible since the left hand integral is stractly greater than $\int_{\mathbb{B}} h_{\epsilon} (x) dx$. This result is proved.
\end{proof}

We shall now discuss the convergence of $c_{\epsilon} u_{\epsilon}$ under the assumption $c_{\epsilon} \rightarrow + \infty$.

\begin{lemma} \label{lem: 2.6}
	The sequence satisfies the following relations
	\begin{flalign}
		& c_{\epsilon} u_{\epsilon} \rightharpoonup G_0, ~ W_0^{1, q} (\mathbb{B}), ~ 1< q < 2, \notag \\
		&c_{\epsilon} u_{\epsilon} \rightarrow G_0, ~ L^p (\mathbb{B}), ~ 1 < p< \frac{2q}{2-q}, \notag \\
		&c_{\epsilon } u_{\epsilon} \rightarrow G_0,~ C_{loc}^1 (\mathbb{B} \setminus \{0\}). \label{eq: 2.9}
	\end{flalign}
    Here $G_0$ is a Green's function and satisfies $- \Delta  G_0 = \delta_0$ in the distributional sense, where $\delta_0$ is the usual Dirac measure centered at $0$.
\end{lemma}

\begin{proof}
	Let 
	\begin{flalign*}
		g_{\epsilon} (x) = \lambda_{\epsilon}^{-1} h_{\epsilon} (x) c_{\epsilon} u_{\epsilon} e^{ 4 \pi ( 1 + \epsilon) u_{\epsilon}^2 }.
	\end{flalign*}
    For any $\varphi \in C_0^{\infty} ( \mathbb{B} )$, we have 
    \begin{flalign*}
    	\underset{ \epsilon \rightarrow 0}{\lim } \int_{ \mathbb{B}} g_{\epsilon} (x) \varphi(x) dx = \varphi(0). 
    \end{flalign*}
    We divide the integral into two parts, i.e,
    \begin{flalign*}
    	\int_{\mathbb{B}} g_{\epsilon} (x) \varphi (x) dx = \int_{u_{\epsilon } \leq \iota c_{\epsilon} } g_{\epsilon} (x) \varphi (x) dx + \int_{u_{\epsilon } > \iota c_{\epsilon}} g_{\epsilon} (x) \varphi (x) dx.
    \end{flalign*}
    We estimate the two integrals on the above equation. For the first part, we obtain
    \begin{flalign*}
    	\int_{ u_{\epsilon} \leq \iota c_{\epsilon} } g_{\epsilon} (x) \varphi (x) dx & = \int_{ u_{\epsilon} \leq \iota c_{\epsilon} } \lambda_{\epsilon}^{-1} h_{\epsilon} (x) c_{\epsilon} u_{\epsilon} e^{ 4 \pi ( 1 + \epsilon) u_{\epsilon}^2 } \varphi(x) dx\\
    	& \leq \frac{ c_{\epsilon} }{ \lambda_{\epsilon} } \left(\underset{\mathbb{B}}{\sup} \varphi(x) \right) \int_{ \mathbb{B} } h_{\epsilon} (x) u_{\epsilon} e^{ 4 \pi ( 1 + \epsilon) u_{\epsilon}^2 } dx.
    \end{flalign*}
    Note that $u_{\epsilon} \rightarrow u_0$ strongly in $L^p (\mathbb{B})$ for $p>1$, then by H\"{o}lder inequality, we have
    \begin{flalign*}
    	& \int_{ \mathbb{B} } h_{\epsilon} (x) u_{\epsilon} e^{ 4 \pi ( 1 + \epsilon) u_{\epsilon}^2 } dx\\
    	& \leq \left( \int_{ \mathbb{B} } h_{\epsilon} (x) e^{ \alpha_{\epsilon} p ( 1 + \epsilon) u_{\epsilon}^2 } dx \right)^{ \frac{1}{p} } \left( \int_{ \mathbb{B} } h_{\epsilon} (x) |u_{\epsilon}|^q   dx \right)^{ \frac{1}{q} } \leq C,
    \end{flalign*}
    where $\frac{1}{p} + \frac{1}{q} =1$. According to the lemma above, we have
    \begin{flalign*}
    	\int_{ u_{\epsilon} \leq \iota c_{\epsilon} } g_{\epsilon} (x) \varphi (x) dx = o_{\epsilon} (1).
    \end{flalign*}
    By the lemma \ref{lem: 2.6}, it follows that $\mathbb{B}_{R_{ \epsilon }^{ \frac{1}{1+ \epsilon} } } \subset \{ u_{\epsilon} > \iota \epsilon \}$ for $\epsilon$ small enough. Then we can implies 
    \begin{flalign*}
    	\int_{ \mathbb{B}_{R_{ \epsilon }^{ \frac{1}{1+ \epsilon} } } \cap \{ u_{\epsilon} > \iota \epsilon \} } g_{\epsilon} (x) \varphi (x) dx & = \int_{ \mathbb{B}_{R_{ \epsilon }^{ \frac{1}{1+ \epsilon} } } } g_{\epsilon} (x) \varphi (x) dx\\
    	& = \int_{ \mathbb{B}_{R_{ \epsilon }^{ \frac{1}{1+ \epsilon} } } } \lambda_{\epsilon}^{-1} h_{\epsilon} (x) c_{\epsilon} u_{\epsilon} e^{ 4 \pi ( 1 + \epsilon) u_{\epsilon}^2 } \varphi (x) dx\\
    	& = c_{\epsilon}^{-1} 
    	\int_{\mathbb{B}_R } h_{\epsilon } (y) u_{\epsilon} (y) e^{ 4 \pi (1+ \epsilon) \left( u_{\epsilon}(y)^2 - c_{\epsilon}^2 \right) }  \varphi (y)  dy + o_{\epsilon}(1).
    \end{flalign*}
    From the remark \ref{rem: 2.1}, we have that
    \begin{flalign*}
    	\int_{ \mathbb{B}_{R_{ \epsilon }^{ \frac{1}{1+ \epsilon} } } \cap \{ u_{\epsilon} > \iota \epsilon \} } g_{\epsilon} (x) \varphi (x) dx & = \varphi(0) \left( 1 + o_{\epsilon} (1) \right) \left( 1 + o_{R} (1) \right) \\
    	& =  \varphi(0) \left( 1 + o_{\epsilon} (1) + o_{R} (1) \right).
    \end{flalign*}
    On the other hand, we obtain
    \begin{flalign*}
    	\int_{\{ u_{\epsilon} > \iota c_{\epsilon} \}\setminus \mathbb{B}_{ R_{ r_{\epsilon}^{ \frac{1}{1 + \epsilon} } } }  }  g_{\epsilon} (x) \varphi(x) dx &  \leq \frac{1}{\iota}
    	\left( \underset{\mathbb{B}}{\sup} |\varphi(x)| \right) \lambda_{\epsilon}^{-1} \cdot\\
    	& \int_{\{ u_{\epsilon} > \iota c_{\epsilon} \}\setminus \mathbb{B}_{ R_{ r_{\epsilon}^{ \frac{1}{1 + \epsilon} } } }  } h_{\epsilon}(x) c_{\epsilon} u_{\epsilon} e^{ 4 \pi ( 1 + \epsilon) u_{\epsilon}^2 } dx \\
    	& \leq \frac{1}{\iota}
    	\left( \underset{\mathbb{B}}{\sup} |\varphi(x)| \right) \left( 1- \int_{\mathbb{B}_R} |x|^{2 \epsilon} e^{ 8 \pi  \phi }dx + o_{\epsilon} (1)    \right)\\
    	& = o_{\epsilon} (1) + o_R (1).
    \end{flalign*}  
    We have
    \begin{flalign*}
    	\underset{\epsilon \rightarrow 0}{\lim} \int_{u_{\epsilon} > \iota c_{\epsilon}} g_{\epsilon} (x) \varphi (x) dx = \varphi(0),
    \end{flalign*}
    by $R \rightarrow + \infty$.
    Then $\varphi(x)$ is the Dirac function. Multiplying both sides of the equation \eqref{eq: 2.1} by $c_{\epsilon}$, we have 
    \begin{flalign*}
    	- \Delta c_{\epsilon} u_{\epsilon} = \frac{1}{\lambda_{\epsilon}} h_{\epsilon} (x)c_{\epsilon} u_{\epsilon}  e^{ 4 \pi (1 + \epsilon) u_{\epsilon}^2 }, ~ \text{in} ~ \mathbb{B}.
    \end{flalign*}
    Then we conclude that $g_{\epsilon} (x)$ is bounded in $L^1_{loc}(\mathbb{B})$. 
    Applying the results above, similar methods of proof in \cite{10}, and the Sobolev embedding
    theorem, we obtain a subsequence $\{ u_{\epsilon} \}_{\epsilon}$ satisfying 
    \begin{equation*}
    	\begin{cases}
    		& c_{\epsilon} u_{\epsilon} \rightharpoonup G_0, ~ \text{in}
    		~ W_0^{1,q}(\mathbb{B}), ~ \forall 1 < q < 2;		\\
    		&c_{\epsilon}
    		 u_{\epsilon} \rightarrow G_0, ~ \text{in}
    		~ L^{s}(\mathbb{B}), ~ \forall 1 \leq s \leq \frac{2q}{2-q};	\\
    		&c_{\epsilon}
    	    u_{\epsilon} \rightarrow G_0, ~ \text{in}
    		~ C^1_{loc} \left(\overline{\mathbb{B}} \setminus \{0\} \right),
    	\end{cases}
    \end{equation*}
    where $G_0$ is a weak solution of 
    \begin{equation*}
    	- \Delta G_0 = \delta_0 ~ \text{in} ~ \mathbb{B}.
    \end{equation*}

    We proceed to analyze the convergence of $u_{\epsilon} $ away from zero. Moreover,  $G_0$ takes the form 
    \begin{flalign*}
    	G_0 = - \frac{1}{ 2 \pi } \ln |x| + A_0 + \beta (x)
    \end{flalign*}
     by \cite{11}, where $A_0$ is a constant, $\beta (x) \in C^{1, \gamma} (\mathbb(\overline{B}) )$ for $0 < \gamma <1$, $\gamma(0) = 0 $.
\end{proof}
    
\begin{lemma}
     Let $\zeta_{\epsilon} (x) \in W_0^{1,2} ( \mathbb{B} )$ with $\int_{\mathbb{B} } |\zeta_{\epsilon} (x)|^2 dx \leq 1$, $ \zeta_{\epsilon} (x) \rightharpoonup 0$ weakly in $W_0^{1,2} ( \mathbb{B} )$ as $\epsilon \rightarrow 0$. $ \zeta_{\epsilon} (x) $ is nonnegative and radially symmetric. Then we have 
     \begin{flalign*}
     	\underset{\epsilon \rightarrow 0}{ \lim \sup } \int_{\mathbb{B}} h_{\epsilon} (x) \left( e^{ 4 \pi (1 + \epsilon) \zeta_{\epsilon}^2 } -1 \right) dx \leq \pi e.
     \end{flalign*}
\end{lemma}
 
\begin{proof}
     By the radial symmetry of $\zeta_{\epsilon} (x)$, we have $ \zeta_{\epsilon} (x) = \zeta_{\epsilon} (r)$ with $r=|x|$, and we make the change of variables,
     \begin{flalign*}
     	\rho_{\epsilon} (r) = (1 + \epsilon)^{ \frac{1}{2} } \zeta_{\epsilon} \left( r^{ \frac{1}{ 1 + \epsilon} } \right).
     \end{flalign*}
    A straightforward calculation shows that
     \begin{flalign*}
     	\int_{\mathbb{B}} | \nabla \rho_{\epsilon} (x)|^2 dx & = 2 \pi \int_0^1 | \nabla (1 + \epsilon)^{ \frac{1}{2} } \zeta_{\epsilon} \left( s \right)|^2 rdr\\
     	& = 2 \pi \int_0^1 \zeta_{\epsilon}^{ \prime } (r) ^2 r dr\\
     	& = \int_{\mathbb{B}} | \nabla \zeta_{\epsilon} \left( x \right)|^2 dx.
     \end{flalign*}
     Then above estimate implies that $|| \rho_{\epsilon} ||_{ W_0^{1, 2}  \left( \mathbb{B} \right) }=   || \zeta_{\epsilon} ||_{ W_0^{1, 2}  \left( \mathbb{B} \right) } \leq 1$. Passing to a subsequence, we may assume that 
     \begin{equation*}
     	\begin{cases}
     		& \rho_{\epsilon} \rightharpoonup \rho_* , ~ W_0^{1,2} \left( \mathbb{B} \right);\\
     		& \rho_{\epsilon} \rightarrow \rho_* , ~ L^p \left( \mathbb{B}  \right), ~ \forall p > 0;\\
     		& \rho_{\epsilon} \rightarrow \rho_* , ~ \text{a.e.} ~ \mathbb{B}.
     	\end{cases}
     \end{equation*}
     It is obvious to see $\zeta_{\epsilon} \rightarrow 0$ a.e. in $\mathbb{B}$, and then it follows that $ \rho_* \rightarrow 0$ a.e. in $\mathbb{B}$. Furthermore, by a change of variable $r = \mu^{ \frac{1}{1 + \epsilon} }$, it holds that
     \begin{flalign*}
     	\int_{ \mathbb{B}} h_{\epsilon } (x) \left\lbrace e^{ 4 \pi \left( 1 + \epsilon \right) \zeta^2_{ \epsilon } (x)  } - 1 \right\rbrace dx  & = 2 \pi \int_0^1 h_{\epsilon} (r) r \left\lbrace e^{ 4 \pi \left( 1 + \epsilon \right) \zeta^2_{ \epsilon }
        (r)  } - 1 \right\rbrace dr\\
         & \leq \frac{1}{1 + \epsilon} \int_{\mathbb{B}} h_{\epsilon} (x) \left\lbrace e^{ 4 \pi \rho_{\epsilon}^2 (x)   }-1 \right\rbrace dx\\
         & \leq \frac{1}{1+ \epsilon} \pi e ,
     \end{flalign*}
     where $A_0$ is contained in the Green function. It follows from that a result of Yang's reference \cite{10} that
     \begin{flalign*}
     	\underset{\epsilon \rightarrow 0}{\lim \sup} \int_{ \mathbb{B}} h_{\epsilon}(x) e^{4 \pi  u_{\epsilon}^2} dx \leq  \int_{ \mathbb{B}} h_{0} (x) dx + \pi e.
     \end{flalign*}
     Also we have 
     \begin{flalign*}
     	\underset{\epsilon \rightarrow 0}{\lim \sup} \int_{ \mathbb{B}} h_{\epsilon } (x) \left\lbrace e^{ 4 \pi \left( 1 + \epsilon \right) \zeta^2_{ \epsilon } (x)  } - 1 \right\rbrace dx  \leq \pi e .
     \end{flalign*}
\end{proof}

\begin{lemma} \label{lem: 2.8}
	There holds 
	\begin{flalign*}
		\underset{u \in W_0^{1,2} (\mathbb{B}), ||u||_{ W_0^{1,2} ( \mathbb{B} ) } \leq 1 }{\sup} \int_{ \mathbb{B}} h_{\epsilon} (x) e^{ 4 \pi u^2}  dx \leq \int_{\mathbb{B}} h_{0} (x) dx +  \pi e^{1 + 4 \pi A_0} .
	\end{flalign*}
\end{lemma}

\begin{proof}
	Let $\mathbf{n}$ be the unit outward normal to $\partial \mathbb{B}_{\delta}$ as $0< \delta <1$. We see that
	\begin{flalign*}
		\int_{ \mathbb{B} \setminus \mathbb{B}_{\delta} } G_0 \cdot \Delta G_0 dx = 0.
	\end{flalign*}
    By the divergence theorem, we have that
    \begin{flalign*}
    	 \int_{ \mathbb{B} \setminus \mathbb{B}_{\delta} } |\nabla G_0|^2 dx & =  \int_{\partial  \left( \mathbb{B} \setminus \mathbb{B}_{\delta} \right) } G_0 \frac{ \partial G_0 }{ \partial \mathbf{n}} ds - \int_{ \mathbb{B} \setminus \mathbb{B}_{\delta} } G_0 \Delta G_0 dx \\
    	 & = G_0 ( \delta ).
    \end{flalign*}
    From the definition of Green function, we obtain that
    \begin{flalign*}
    	G_0 (\delta) = - \frac{1}{2 \pi } \ln \delta + A_0 + o_{\delta} (1),
    \end{flalign*}
    since $G_0$ is a weak solution of $- \Delta G_0 = \delta_0$.
    Hence, we obtain that 
    \begin{flalign*}
    	\int_{ \mathbb{B} \setminus \mathbb{B}_{\delta}} | \nabla G_0|^2 dx = - \frac{1}{2 \pi} \ln \delta + A_0 + o_{\delta} (1) + o_{\epsilon} (1).
    \end{flalign*}
    With the aid of \eqref{eq: 2.9}, we have that
    \begin{flalign*}
    	\int_{ \mathbb{B} \setminus \mathbb{B}_{\delta}} | \nabla u_{\epsilon}|^2 dx = c_{\epsilon}^{- 2 } \int_{ \mathbb{B} \setminus \mathbb{B}_{\delta}} \left( | \nabla G_0|^2 + o_{\epsilon} (1) \right) dx\\
    	= c_{\epsilon}^{- 2 }  \left(- \frac{1}{2 \pi } \ln \delta + A_0 + o_{\delta} (1) + o_{\epsilon} (1) \right).
    \end{flalign*}
     Denote $\kappa_{\epsilon} = \int_{  \mathbb{B}_{\delta}} | \nabla u_{\epsilon}|^2 dx $, it follows that
     \begin{flalign*}
     	\kappa_{\epsilon} & = 1 - \int_{ \mathbb{B} \setminus \mathbb{B}_{\delta}} | \nabla u_{\epsilon}|^2 dx \\
     	& = 1- c_{\epsilon}^{-2 }  \left(- \frac{1}{ 2 \pi} \ln \delta + A_0 + o_{\delta} (1) + o_{\epsilon} (1) \right).
     \end{flalign*}
     We set $s_{\epsilon} = \underset{\partial \mathbb{B}_{\delta}}{\sup}  u_{\epsilon} = u_{\epsilon}  ( \delta ) $ and $\tilde{u}_{\epsilon} = \left( u_{\epsilon} - s_{\epsilon} \right)^+ $. Obviously, $s_{\epsilon} = c_{\epsilon}^{ -1 } \left( \underset{\partial \mathbb{B}_{\delta} }{\sup } G_0 + o_{\epsilon} (1) \right)$, $\tilde{u}_{\epsilon} \in W_0^{1,2} \left( \mathbb{B}_{\delta} \right)$ and $\int_{\mathbb{B}_{\delta}} | \nabla  \tilde{u}_{\epsilon}|^2 dx \leq  \int_{\mathbb{B}_{\delta}} | \nabla u_{\epsilon} |^2 dx $. Then we have that
     \begin{flalign*}
     	4 \pi \left(1 + \epsilon \right)u_{\epsilon}^2 & \leq 4 \pi \left( 1+ \epsilon \right) \left( s_{\epsilon}  + \tilde{u}_{\epsilon} \right)^2.
     \end{flalign*} 
     For sufficiently small $ \epsilon$, we have  $\mathbb{B}_{R r_{\epsilon}^{ \frac{1}{1+ \epsilon} }   } \subset \mathbb{B}_{\delta}$, and hence
     \begin{flalign}
     	\int_{\mathbb{B}_{R r_{\epsilon}^{ \frac{1}{1+ \epsilon} }   }} 
     	h_{\epsilon } (x) e^{ 4 \pi \left(1 + \epsilon \right) \frac{ \tilde{u}_{\epsilon}^2 }{\kappa_{\epsilon} }  } dx & \leq \int_{\mathbb{B}_{R r_{\epsilon}^{ \frac{1}{1+ \epsilon} }   }} 
     	h_{\epsilon }(x) \left(e^{ 4 \pi \left(1 + \epsilon \right) \frac{ \tilde{u}_{\epsilon}^2 }{\kappa_{\epsilon} }  }- 1 \right) dx + o_{\epsilon} (1) \notag  \\
     	& \leq \int_{\mathbb{B}_{\delta   }} 
     	h_{\epsilon } (x) \left( e^{ 4 \pi \left(1 + \epsilon \right) \frac{ \tilde{u}_{\epsilon}^2 }{\kappa_{\epsilon} }  }- 1 \right) dx + o_{\epsilon} (1).\label{eq: 2.10}
     \end{flalign}
     Also we have that
     \begin{flalign}
     	4 \pi \left(1 + \epsilon \right)u_{\epsilon}^2 & \leq 4 \pi \left( 1+ \epsilon \right) \left( s_{\epsilon}  + \tilde{u}_{\epsilon} \right)^2 \notag \\
     	& \leq 4 \pi (1+ \epsilon) \left( c_{\epsilon}^{-2} G^2 (\delta) + 2 c_{\epsilon}^{-1} G (\delta) \tilde{u}_{\epsilon} +  \tilde{u}_{\epsilon}^2 \right)  + o_{\epsilon} (1)\notag \\
     	& \leq 4 \pi (1+ \epsilon) \left( 2 G (\delta) + \tilde{u}^2_{\epsilon} \right)  + o_{\epsilon} (1) \notag \\
     	& \leq 4 \pi (1+ \epsilon) \frac{ \tilde{u}_{\epsilon}^2 }{\kappa_{\epsilon}} - 2 (1+ \epsilon) \ln \delta + 4 \pi (1 + \epsilon) A_0 + o_{\epsilon}(1).\label{eq: 2.11}
     \end{flalign} 
     From equations \eqref{eq: 2.10} and \eqref{eq: 2.11} together with  the lemma \ref{lem: 2.8}, we obtain
     \begin{flalign}
     	\int_{\mathbb{B}_{R r_{\epsilon}^{ \frac{1}{1+ \epsilon} }   }} 
     	h_{\epsilon } (x) e^{ 4 \pi  \left(1 + \epsilon \right) u_{\epsilon}^2   } dx & \leq \int_{\mathbb{B}_{R r_{\epsilon}^{ \frac{1}{1+ \epsilon} }   }} 
     	h_{\epsilon } (x) e^{ 4 \pi \left( 1+ \epsilon \right) \left( s_{\epsilon}  + \tilde{u}_{\epsilon} \right)^2  } dx \notag  \notag \\
     	& \leq \int_{\mathbb{B}_{R r_{\epsilon}^{ \frac{1}{1+ \epsilon} }   }} 
     	h_{\epsilon } (x) e^{ 4 \pi \left( 1+ \epsilon \right) \left( \frac{\tilde{u}_{\epsilon}^2  }{ \kappa_{\epsilon} } - \frac{1}{2 \pi} \ln \delta + A_0 + o_{\epsilon}(1) \right) } dx \notag  \\
     	& \leq \delta^{ -2(1+ \epsilon) } e^{ 4 \pi (1+ \epsilon) ( A_0 + o_{\epsilon}(1) )  } \notag  \\
     	& \left( \int_{\mathbb{B}_{\delta   }} 
     	|x|^{2 \epsilon}  \left(e^{ 4 \pi \left(1 + \epsilon \right) \frac{ \tilde{u}_{\epsilon}^2 }{\kappa_{\epsilon} }  } - 1 \right) dx + o_{\epsilon} (1) \right) \notag \\
     	& \leq \delta^{ -2 \epsilon } \pi e^{ 1 + 4 \pi A_0 + o_{\epsilon}(1) } + o_{\epsilon}(\delta), \label{eq: 2.12}
     \end{flalign}
     with $o_{\epsilon} (\delta) \rightarrow 0$ as $\epsilon \rightarrow 0$ for any fixed $\delta > 0$. Letting $\epsilon \rightarrow 0$ firstly and $R \rightarrow + \infty$ laterly in \eqref{eq: 2.12}, we have
     \begin{flalign*}
     	& \underset{R \rightarrow + \infty}{\lim} \underset{\epsilon \rightarrow 0}{\lim \sup} \int_{\mathbb{B}_{R r_{\epsilon}^{ \frac{1}{1+ \epsilon} }   }} 
     	h_{\epsilon } (x)e^{ 4 \pi \left(1 + \epsilon \right) u_{\epsilon}^2   } dx  \leq \pi e^{1 + 4 \pi A_0}.
     \end{flalign*}
     Then we conclude that
     \begin{flalign*}
     	\underset{\epsilon \rightarrow 0}{\lim \sup} \int_{\mathbb{B}} h_{\epsilon} (x) e^{ 4 \pi (1+ \epsilon) u_{\epsilon}^2}   dx \leq \int_{\mathbb{B}} h_{0} (x) dx +  \pi e^{1 + 4 \pi A_0}.
     \end{flalign*}
     The proof is completed.
\end{proof}

\section{Exclusion of Blow-up Phenomenon}

In this section, we construct a sequence of functions $\phi_{\epsilon} (x) \in W_0^{1,2} \left( \mathbb{B} \right) $ to show that
\begin{flalign*}
	 \int_{\mathbb{B}} h_{\epsilon} (x) e^{ 4 \pi (1+ \epsilon) \phi_{\epsilon}^2 } dx > \int_{\mathbb{B}} h_{0} (x) dx +  \pi  e^{1 + 4 \pi A_0}.
\end{flalign*}
From \cite{15}, we set 
\begin{equation*}
	\phi_{\epsilon} (x)= 
	\begin{cases}
		& c+ c^{-1} \left( - \frac{ 1  }{4 \pi} \ln \left( 1 + \pi |\frac{x}{\epsilon}|^2 \right) + b \right), ~ |x| \leq R \epsilon;\\
		& \frac{G}{c}, ~ R \epsilon \leq |x|\leq 1 .
	\end{cases}
\end{equation*}
In order to ensure that $\phi_{k}(x) \in W_0^{1,2} \left( \mathbb{B} \right)$, we require that
\begin{flalign*}
	c+ c^{-1} \left( - \frac{1 }{4 \pi} \ln \left( 1 + \pi |R|^2 \right) + b \right) = \frac{G \left( R \epsilon \right)}{c},
\end{flalign*}
which imples that
\begin{flalign*}
	c^2= G \left( R \epsilon \right) + \frac{1}{4 \pi } \ln \left( 1 + \pi | R|^2 \right) - b.
\end{flalign*}

From $\int_{\mathbb{B}} | \nabla \phi_{\epsilon}(x)|^2 dx =1$, we have 
\begin{flalign*}
	\int_{|x| \leq R \epsilon } | \nabla \phi_{\epsilon} (x)|^2 dx 	& = 2 c^{ - 2 } \left\lbrace \ln \left( 1 + \pi R^2 \right) -  1 + O\left( R^{-2 } \right)  \right\rbrace,
\end{flalign*}
and
\begin{flalign*}
	\int_{ R \epsilon \leq |x| \leq 1 } | \nabla \phi_{\epsilon}|^2 dx = c^{- 2} \left( - \frac{1}{2 \pi} \ln |R \epsilon| + A_0 + O\left( (R \epsilon)^{2 \gamma} \right) \right),
\end{flalign*}
which together imply
\begin{flalign}
	c^{2} = 2 \left( \ln \pi + 2 \ln R - 1 + O \left( R^2 \right) \right) + G(R \epsilon). \label{eq: 3.1}
\end{flalign}
On the other hand, we obtain
\begin{flalign}
	b= - c^2 + G( R \epsilon) + \frac{1}{4 \pi } \ln \left( 1+ \pi R^2 \right). \label{eq: 3.2}
\end{flalign}
In $\mathbb{B}_{ R \epsilon}$, we have
\begin{flalign} \label{eq: 3.3}
	4 \pi  \phi_{\epsilon}^2 & = 4 \pi c^{2} \left(  1 + \frac{1}{ c^2  } \left( - \frac{1}{ 4 \pi} \ln \left( 1 + \pi |\frac{x}{\epsilon}|^2 \right) + b \right) \right)^2 \notag \\
	& \geq 4 \pi c^2 \left(  1 + \frac{1}{ c^2  } \left( - \frac{1}{ 4 \pi } \ln \left( 1 + \pi |\frac{x}{\epsilon}|^2 \right) + b \right) \right) \notag \\
	& = 4 \pi c^2 + 8 \pi b - 2 \ln \left( 1 + \pi | \frac{x}{\epsilon}|^2 \right).
\end{flalign}
We consider the right-hand side of the equation \eqref{eq: 3.3},
\begin{flalign*}
	4 \pi c^2 + 8 \pi b \geq 2R + 1 + \ln \pi + 4 \pi A_0 + O \left( R^{ -2 } \right),
\end{flalign*}
since
\begin{flalign*}
	4 \pi \left( c^2 + 2 b \right) & = 4 \pi \left\lbrace c^2 + 2 \left( - c^2 + G( R \epsilon ) + \frac{1}{4 \pi} \ln \left( 1+ \pi R^2  \right) \right) \right\rbrace \\
	& = - 2 \left\lbrace \ln \left( 1+ \pi R^2\right)  - 1 + O \left( R^{ - 2 }\right)  \right\rbrace + 4 \pi G \left( R \epsilon \right) + 2 \ln \left(  1 + \pi R^2 \right)\\
	& \geq - 2 \ln \left( 1 + \pi R^2 \right) + 1 + O \left( R^{- 2  } \right) + 4 \pi A_0 \\
	& \geq - 2 \ln \left( 1 + \pi R^2 \right) + 1 + O \left( R^{- 2  } \right) - 2 \ln | R \epsilon | + 4 \pi A_0 .
\end{flalign*}
We consider an approprite $R$ and $\epsilon$ such that
\begin{flalign*}
	R^{ - 4 } |\epsilon|^{ -2 } \rightarrow + \infty
\end{flalign*}
as $R \rightarrow + \infty$, $\epsilon \rightarrow 0$, and $R \epsilon \rightarrow 0$. Then
\begin{flalign*}
	e^{- 2 \ln \left( 1+ \pi R^2 \right) - 2 \ln | R \epsilon| } = \left( 1 + \pi R^{2} \right)^{ - 2 } \cdot |R \epsilon|^{-2} \rightarrow + \infty.
\end{flalign*}
Moreover, we have
\begin{flalign}
	\int_{|x| < R \epsilon } h_{\epsilon} (y) e^{ - 2 \ln \left(1 + \pi |\frac{x}{\epsilon} |^2 \right) } dx & = \epsilon^2 \int_{ |y| < R } \frac{ h_{\epsilon}(y) }{ \left( 1+ \pi |y|^2 \right)^2 } dy \notag \\
	& = \epsilon^2 \left( \int_{ |y| < R} h_{\epsilon} (y) dy + O \left( R^{ - 2 } \right) \right). \label{eq: 3.4}
\end{flalign}
Combining equations \eqref{eq: 3.1}, \eqref{eq: 3.2}, \eqref{eq: 3.3} and \eqref{eq: 3.4}  yields 
\begin{flalign*}
	\int_{ |x| < R \epsilon} h_{\epsilon} (x) e^{ 4 \pi \phi_{\epsilon}^2 } dx & \geq \pi e^{1 + 4 \pi A_0 } + O \left( R^{ -2 } \right) .
\end{flalign*}
From $e^t \geq 1+t$, we have
\begin{flalign*}
	\int_{ R \epsilon \leq |x| <1 } h_{\epsilon} (x) e^{ 4 \pi \phi_{\epsilon}^2    } dx> \int_{\mathbb{B}} h_{0} (x) dx + O \left( R^{-2} \right).
\end{flalign*}
From the two contradictory inequalities, the sequence $u_{\epsilon}$ is uniformly bounded. Elliptic estimates of the equation \eqref{eq: 2.1} then yield that $u_{\epsilon}$ converges to the extremal function in $C^1 \left( \overline{\mathbb{B}} \right)$, and hence Theorem 1.1 is proved.

\textbf{Acknowledgments}

We would like to take this opportunity to thank the editor and the
reviewers again for their constructive comments and useful suggestions, and the time and eﬀorts
they have spent in the review process.

\textbf{Conflict of interest statement}

The authors declare no conﬂict of interest.






\end{document}